[1994/12/01]
\documentclass{amsart}
\chardef\bslash=`\\ 

\usepackage{amsmath}
\usepackage{amsthm}
\usepackage{amssymb}
\usepackage[T1]{fontenc}
\usepackage{dsfont}
\usepackage[many]{tcolorbox}
\usepackage{soul}
\usepackage{xcolor}
\usepackage{graphicx}
\usepackage{adjustbox}
\usepackage{tikz-cd}
\usepackage{appendix}
\usepackage{enumerate}
\usepackage{multicol}
\usepackage{fancyhdr}
\usepackage{parskip}
\usepackage{hyperref}
\usepackage{cleveref}
\usepackage{mathtools}
\usepackage{nameref}
\usepackage{thmtools}
\usepackage{thm-restate}
\usepackage{setspace}
\usepackage{float}
\usepackage{bm}
\usepackage{orcidlink}
\usepackage{braket}

\hypersetup{
  colorlinks   = true, 
  urlcolor     = {blue!50!black}, 
  linkcolor    = {blue!50!black}, 
  citecolor   = {red!50!black} 
}

\usepackage{ulem}
\usepackage{todonotes}

\colorlet{NadavColor}{-green!40!yellow}
\colorlet{ArisColor}{red}

\newtheorem{theorem}{Theorem}[section]
\newtheorem*{theorem*}{Theorem}
\newtheorem{lemma}[theorem]{Lemma}
\newtheorem{proposition}[theorem]{Proposition}
\newtheorem*{proposition*}{Proposition}
\newtheorem{corollary}[theorem]{Corollary}
\newtheorem*{corollary*}{Corollary}
\newtheorem*{fact*}{Fact}
\newtheorem{fact}[theorem]{Fact}

\newtheorem{thmx}{Theorem}

\Crefname{thmx}{Theorem}{Theorems}

\theoremstyle{definition}
\newtheorem*{definition*}{Definition}
\newtheorem{definition}[theorem]{Definition}

\newtheorem{question}[theorem]{Question}
\newtheorem*{question*}{Question}
\newtheorem{example}[theorem]{Example}
\newtheorem{note}[theorem]{Note}

\newtheorem*{notation*}{Notation}

\theoremstyle{remark}
\newtheorem{remark}[theorem]{Remark}

\newcommand{\limplies}{\ensuremath{\rightarrow}}

\newcommand{\into}{\ensuremath{\hookrightarrow}}

\def\Ind#1#2{#1\setbox0=\hbox{$#1x$}\kern\wd0\hbox to 0pt{\hss$#1\mid$\hss}
\lower.9\ht0\hbox to 0pt{\hss$#1\smile$\hss}\kern\wd0}

\def\notind#1#2{#1\setbox0=\hbox{$#1x$}\kern\wd0
\hbox to 0pt{\mathchardef\nn=12854\hss$#1\nn$\kern1.4\wd0\hss}
\hbox to 0pt{\hss$#1\mid$\hss}\lower.9\ht0 \hbox to 0pt{\hss$#1\smile$\hss}\kern\wd0}

\newcommand{\N}{\mathbb{N}}
\newcommand{\Z}{\mathbb{Z}}

\newcommand{\bN}{\mathbb{N}}
\newcommand{\bM}{\mathbb{M}}

\newcommand{\M}{\mathcal{M}}
\newcommand{\C}{\mathcal{C}}
\newcommand{\Ncal}{\mathcal{N}}
\newcommand{\Mcal}{\mathcal{M}}
\newcommand{\Mbb}{\mathbb{M}}
\renewcommand{\L}{\mathcal{L}}
\newcommand{\Th}{\mathsf{Th}}

\newcommand{\tp}{\mathsf{tp}}
\newcommand{\qftp}{\mathsf{qftp}}

\newcommand{\cA}{\mathcal{A}}
\newcommand{\cB}{\mathcal{B}}
\newcommand{\cC}{\mathcal{C}}

\newcommand{\cL}{\mathcal{L}}
\newcommand{\cM}{\mathcal{M}}
\newcommand{\cN}{\mathcal{N}}

\let\lL\cL

\newcommand{\LO}{\mathsf{LO}}

\newcommand{\Flim}{\ensuremath{\mathsf{Flim}}}

\newcommand{\restr}{\upharpoonright}

\newcommand{\arrows}[2]{#1\rightarrow(#2)}

\DeclareMathOperator{\Age}{\mathsf{Age}}

\DeclareMathOperator{\Aut}{\mathsf{Aut}}

\DeclareMathOperator{\ind}{\mathsf{Ind}}

\let\age\Age

\let\vphi\varphi

\newcommand{\qford}{$\L_{\infty,0}$-orderable}
\renewcommand{\Im}{\mathsf{Im}}

\title[Infinitary Expansions]{Practical \& Structural Infinitary Expansions}
\date{\today}

\subjclass[2020]{Primary: 05C55, 03C75 Secondary: 03C52 }

\keywords{Generalised Indiscernibles, Ramsey class, Modelling Property, Infinitary Logic}

\author[N. Meir]{Nadav Meir \orcidlink{0000-0002-7774-2892}}
\address{Instytut Matematyczny,\\ Uniwersytet Wroc{\l}awski,\\ pl. Grunwaldzki 2,\\ 50-384 Wroc{\l}aw, \\ Poland}
\email{\href{mailto:nadavmeir@gmail.com}{nadavmeir@gmail.com}}
\author[A. Papadopoulos]{Aris Papadopoulos\orcidlink{0000-0001-7071-4277}}

\address{School of Mathematics, \\University of Leeds, \\Leeds LS2 9JT, \\ United Kingdom}
\email{\href{mailto:mmadp@leeds.ac.uk}{mmadp@leeds.ac.uk}}

\thanks{Meir is supported by Narodowe Centrum Nauki, Poland, grant 2016/22/E/ST1/00450. Papadopoulos is supported by a Leeds Doctoral Scholarship, from the University of Leeds. This research is part of his Ph.D. project} 

\begin{document}

\maketitle

\begin{abstract}
    Given a structure $\M$ we introduce infinitary logic expansions, which generalise the Morleyisation. We show that these expansions are tame, in the sense that they preserve and reflect both the Embedding Ramsey Property (ERP) and the Modelling Property (MP). We then turn our attention to Scow's theorem connecting generalised indiscernibles with Ramsey classes and show that by passing through infinitary logic, one can obtain a stronger result, which does not require any technical assumptions. We also show that every structure with ERP, not necessarily countable, admits a linear order which is a union of quantifier-free types, effectively proving that any Ramsey structure is ``essentially'' ordered. We also introduce a version of ERP for classes of structures which are not necessarily finite (the finitary-ERP) and prove a strengthening of the Kechris-Pestov-Todorcevic correspondence for this notion.
\end{abstract}

\section{Introduction}\label{sec:introduction}

Structural Ramsey theory focuses on classes of structures which satisfy a ``structural'' analogue of the conclusion of Ramsey's Theorem. Such classes are called Ramsey classes (see \Cref{def:hp-jep-ap-erp}(4), and \Cref{def:rc} for precise definitions). In the terminology of structural Ramsey theory, Ramsey's theorem says that the class of all finite linear orders is a Ramsey class. The study of various properties of Ramsey classes has been a central theme in twenty-first-century combinatorics, being revitalised by the surprising and beautiful results of \cite{KechrisPestovTodorcevic_2005}, connecting Ramsey classes with topological dynamics. 

On the other hand, in recent years, generalised indiscernibles have turned out to be a central tool in modern model theory, in particular, in classification theory beyond stable structure. More precisely, a sequence of (same-length) tuples in some $\L$-structure $\M$, indexed by elements of an $\L^\prime$-structure $\Ncal$ is an \emph{$\Ncal$-indexed indiscernible} if whenever elements in the indexing set have the same quantifier-free type in $\Ncal$, then the tuples they index have the same type in $\M$. This is a vast generalisation of the concept of an order-indiscernible sequence, which in this terminology can be recast as an $\Ncal$-indexed indiscernible sequence, where $\Ncal$ is any infinite linearly ordered set. A crucial property of order-indiscernible sequences is the following: For any infinite linear order $(\Ncal,\leq)$, whenever we are given an $(\Ncal,\leq)$-indexed sequence in some sufficiently saturated $\L$-structure we can find an $(\Ncal,\leq)$-indexed indiscernible sequence realising its Ehrenfeucht–Mostowski type (EM-type). The generalisation of this property for arbitrary indexing structures is known as the \emph{modelling property} (defined in \Cref{def:mp}). Applications of generalised indiscernibles can be traced back to the foundational work of Shelah in classification theory \cite{Shelah_1990}. Returning to current trends in model theory, generalised indiscernibles have found applications in model-theoretic tree property arguments (for example, as in \cite{ChernikovRamsey_2016}). Moreover, it has been shown that many classification-theoretic dividing lines can be characterised by a collapse of one kind of generalised indiscernibles to another; see, for instance, \cite{Scow_2012, GuingonaHillScow_2017, ChernikovPalacinTakeuchi_2019}. 

The two lines of research outlined above may seem, at first sight, disconnected from each other, but they are, in fact, closely intertwined. In his original paper \cite{Ramsey_1930}, Ramsey used his theorem to show that certain fragments of first-order logic were decidable, essentially through a weak form of indiscernibility (though not under that name). Moreover, the fact that given any sequence indexed by a linear order, one can always realise its EM-type by an indiscernible sequence is essentially an application of Ramsey's Theorem and compactness. More recently, Scow showed that the connection between indiscernibility and Ramsey theory is even deeper than what we just outlined. More precisely, the main theorem of \cite{Scow_2015} says that under some technical assumptions, $\Ncal$-indexed indiscernibles have the modelling property if, and only if, $\Age(\Ncal)$ is a Ramsey class. 

One of the main goals of this paper is to strengthen that theorem by showing that most of these assumptions are, in fact, not necessary. More precisely, we show the following theorem:

\begin{thmx}[\Cref{rc mp}]\label{rc mp intro}
        Let $\L^\prime$ be a first-order language, $\C$ a class of finite $\L^\prime$-structures, and $\Ncal$ an infinite, locally finite $\L^\prime$-structure such that $\Age(\Ncal) = \C$. Then, the following are equivalent:
    \begin{enumerate}
        \item $\C$ is a Ramsey class.
        \item $\Ncal$-indexed indiscernibles have the modelling property.
    \end{enumerate}
\end{thmx}

As an almost immediate corollary, we show that every Ramsey class admits an order, given by a union of quantifier-free types, extending a known result for countable Ramsey classes, from \cite{Bodirsky_2015}. More precisely, we obtain the following result:

\begin{thmx}[\Cref{rc qford}]\label{rc qford intro}
        Let $\C$ be a Ramsey class of $\L^\prime$-structures and $\Ncal$ an $\L^\prime$-structure such that $\Age(\Ncal) = \C$. Then there is an $\Aut(\Ncal)$-invariant linear order on $\Ncal$ which is the union of quantifier-free types. More explicitly, there is a (possibly infinite) Boolean combination of atomic and negated atomic $\L^\prime$-formulas $\Phi(x,y):=\bigvee_{i\in I}\bigwedge_{j\in J_i} \vphi_{j_i}^{(-1)^{n_{j_i}}}(x,y)$, such that $\Phi$ is a linear order for every structure in $\C$.
\end{thmx}

Given an $\L^\prime$-structure $\Ncal$, we will say that $\C := \Age(\Ncal)$ is \emph{$\L^\prime_{\infty,0}$-orderable} (see \Cref{def:qford}) if there is a (possibly infinite) Boolean combination of atomic and negated atomic $\L^\prime$-formulas which is a linear order for every structure in $\C$. We observe that this is the minimal required condition such that for every theory $T$ there is some $\Ncal$-indexed indiscernible sequence (in a sufficiently saturated model of $T$), for a more precise statement, see \Cref{existence of inds}, which lists equivalent criteria for the existence of $\Ncal$-indexed indiscernibles. 

In \Cref{sec:infinitary-logic} we introduce a set of techniques based around ``tame'' expansions in infinitary logic, which are the basic tools we use in the proofs of \Cref{rc mp intro,rc qford intro}. These techniques seem to be very fruitful in generalising arguments known to work for relational/finite structures to arbitrary ones.

We also introduce and discuss a generalisation of the Ramsey Property, the \emph{finitary Ramsey Property} (see \Cref{def:f-ERP}), which, we argue, is the natural generalisation of the Ramsey Property to classes of possibly infinite, not necessarily relational structures. To this end, we show, first of all, that \Cref{rc mp intro} holds without the assumption of local finiteness if one replaces ``$\C$ is a Ramsey class'' by ``$\C$ is a finitary Ramsey class'' (see \Cref{frc mp}, for a more precise statement). Moreover, we also obtain the following stronger version of the well-known correspondence of Kechris, Pestov, and Todorcevic:

\begin{thmx}[\Cref{thm:kpt-fERP}]\label{thm:fERP-intro}
    Let $\M$ be an ultrahomogeneous $\L$-structure. Then, the following are equivalent:
    \begin{enumerate}[(1)]
        \item $\Aut(\M)$ is extremely amenable.
        \item $\Age(\M)$ is a finitary Ramsey class.
    \end{enumerate}
\end{thmx}

\subsection*{Structure of the paper} In \Cref{sec:infinitary-logic} we introduce some of the main new machinery that will be used in the remainder of the paper. Then, in \Cref{sec:ramsey classes,sec:gen inds} we cover the basic concepts in structural Ramsey theory and the theory of generalised indiscernibles, respectively. In \Cref{sec:existence} we discuss $\L_{\infty,0}$-orderability as a necessary and sufficient condition for the existence of $\Ncal$-indexed indiscernibles in (sufficiently saturated) models of arbitrary theories. Then, in \Cref{sec:removing ass} we combine all the previous results to prove \Cref{rc mp intro,rc qford intro}. In \Cref{sec:local-finiteness} we discuss classes that no longer necessarily contain finite structures, the finitary Ramsey Property, and prove \Cref{thm:fERP-intro}. Finally, \Cref{sec:questions} contains some closing remarks on further lines of research. 

\subsection*{Background and notation} We assume that the reader is familiar with basic model theory. We recommend \cite[Chapters 1 and 2]{Hodges_1993}, for a general background in model theory, and \cite[Chapter 5.1]{TentZiegler_2012}, for an introduction to indiscernibles. 

\begin{notation*}
    Our notation throughout this paper is standard. Throughout, $\L,\L^\prime$ will denote first-order languages. Typically $\M,\Ncal,\dots$ will denote infinite structures (usually in the languages $\L$ and $\L^\prime$, respectively) with domain $M,N,\dots$, and $A,B,\dots$ will denote finite(ly generated) structures. Given an $\L$-structure $\M$, a tuple $\bar a\in M^n$ and a subset $A\subseteq M$, we denote by $\tp_\M(\bar a/A)$ the type of $\bar a$ over $A$ in $\M$. If $A=\emptyset$, then we just write $\tp_\M(\bar a)$ for $\tp_\M(\bar a/\emptyset)$. We denote by $\qftp_\M(\bar a)\subseteq\tp_\M(\bar a)$ the quantifier-free type of $\bar a$ in $\M$, that is, the set of all quantifier-free $\L$-formulas $\phi(\bar x)$ such that $\M\vDash\phi(\bar a)$. If $\Delta\subseteq\L$ is a set of formulas in free variables from $\bar x$ we write $\tp_\M^\Delta(\bar a)$ for the restriction of $\tp_\M(\bar a)$ to formulas and negations of formulas in $\Delta$, that is, $\tp^\Delta_\M(\bar a) = \{\phi(\bar x)\in \tp_\M(\bar a): \phi(\bar x) \in\Delta\text{ or }\lnot\phi(\bar x)\in\Delta\}$. In general, when we say type (resp. quantifier-free type), we mean, unless otherwise stated, a \emph{complete} type (resp. quantifier-free type).
    
    Unless otherwise stated, when we talk of definable/$\L$-definable sets we mean sets definable \emph{without parameters}. The same convention applies when we discuss (quantifier-free) types. 
\end{notation*} 

\subsection*{Acknowledgements} The research presented in this paper was ignited during Unimod 2022 at the School of Mathematics at the University of Leeds. The authors would like to thank the organisers of the very well-organised programme, the School of Mathematics for its hospitality during the programme, and I. Kaplan for his mini-course which sparked the ideas that evolved into the research presented in this paper. The authors would also like to thank the second author's PhD supervisors P. Eleftheriou and D. Macpherson for their helpful comments.

\section{Expansions in Infinitary Logic}\label{sec:infinitary-logic}

In this \namecref{sec:infinitary-logic}, we introduce our main new machinery, \emph{infinitary Morleyisations} but postpone the actual results to later sections. First, recall that for a language $\L$ and (not necessarily infinite) cardinals $\kappa,\lambda$, we denote by $\L_{\kappa,\lambda}$ the collection of $\L$-formulas of the form:
\[
\underbrace{(Q_1\bar x),\dots,(Q_\alpha\bar x)}_{\alpha<\lambda} \bigwedge_{\beta<\kappa}\left(\bigvee_{\gamma<\kappa}\vphi_{\beta,\gamma}(\bar x)\right),
\]
where each $Q_i$ is a first-order quantifier (i.e. $\forall$ or $\exists$), and each $\vphi_{\beta,\gamma}(\bar x)$ is an $\L$-formula, whose variables $\bar x$ are some finite subtuple of $(x_1,\dots,x_{\alpha})$. By replacing $\kappa$ (resp. $\lambda$) by $\infty$ we mean that the number of conjunctions/disjunctions (resp. quantifiers) is unbounded. In this notation, traditional first-order logic is only concerned with $\L_{\omega,\omega}$-formulas, and $\L_{\infty,0}$ refers to quantifier-free formulas consisting of arbitrarily long conjunctions/disjunctions.

\subsection{Infinitary Morleyisations}\label{sec:infinitary morleyisation}
Modern model theory tends to focus not on the syntactical properties of the definable sets in a given structure (e.g. if they are definable using quantifier-free formulas) but rather on classifying structures, based on various notions of ``tameness'' and ``wildness''. In particular, towards this, it often is easier to work in expansions of a given structure, in which definable sets are easier to describe, but the complexity of the lattice of definable sets remains unchanged (and hence so do the notions of ``tameness'' and ``wildness''). The prototypical example of this process is that of \emph{Morleyisation}, by which we mean the expansion $\widehat{\M}$ of an $\L$-structure $\M$ to a language $\widehat{\L}\supseteq\L$, containing for each $\L$-definable $X\subseteq M^n$ an $n$-ary relation symbol $R_X$, where we naturally interpret the new relation symbols as follows:
\[
\widehat{\M}\vDash R_X(\bar a) \text{ if, and only if, } \bar a\in X,
\]
for all $\L$-definable $X\subseteq M^n$ and all $\bar a\in M^n$. 

In the resulting structure $\widehat{\M}$ all $\L$-definable sets can be defined by quantifier-free formulas, but nonetheless, the actual $\L$-definable sets of the structure remain unchanged. It is an easy observation that many important model-theoretic properties of $\M$ (e.g. categoricity, stability, NIP, and more) are preserved by Morleyisations.

We introduce a more general form of this construction, which as we will show remains ``tame'' in some very well defined sense, but does behave significantly differently than the usual Morleyisation, from the point of view of its first-order theory, and hence from the point of view of Shelah-style classification theory. We start with the most general possible definition:

\begin{definition}
    Let $\kappa,\lambda$ be (not necessarily infinite) cardinals, or $\infty$. We define the \emph{$\L_{\kappa,\lambda}$-Morleyisation} of $\M$ to be the expansion $\widehat{\M}$ of $\M$ to the language:
    \[
    \widehat{\L} := \L\cup \Set{R_X : X \text{ is an $\L_{\kappa,\lambda}$-definable subset of }\M^n, \text{ for } n\in\N},
    \] 
    where we naturally interpret the new relation symbols as follows:
    \[
    \widehat{\M}\vDash R_X(\bar a)\text{ if, and only if } \bar a\in X,
    \]
    for every $\L_{\kappa,\lambda}$-definable subset $X$ of $\M$ and every $\bar a$ in $\M$.
\end{definition}

\begin{remark}
    In our notation, it is clear that the usual Morleyisation we described previously is just the $\L_{\omega,\omega}$-Morleyisation.
\end{remark}

We will be focusing our attention on a special kind of $\L_{\kappa,\lambda}$-Morleyisation, namely the \emph{$\L_{\infty,0}$-Morleyisation}. More precisely, we will be concerned with two special reducts of the $\L_{\infty,0}$-Morleyisation and  we dedicate the following two subsections to these reducts.

\subsection{Quantifier-Free type Morleyisation}\label{sec:QF Type Morleyisation}
It will turn out to be the case that adding predicates for all $\L_{\infty,0}$-definable sets is a somewhat complicated procedure, and for the purposes of discussing generalised indiscernibles, one need no more than the preservation and reflection of equality of quantifier-free types, precisely as stated in \Cref{qf-type-Morleyzation-same-types,qf-isolator-same-types}. This is achieved in the following way:

\begin{definition}\label{def:qf-type-Morleyisation}
    Let $\M$ be an $\L$-structure. The \emph{quantifier-free type Morleyisation} of $\M$ is the expansion $\widehat{\M}_\mathsf{qfi}$ of $\M$ to the language 
    \[
    \widehat{\L}_\mathsf{qfi} := \cL\cup \Set{R_p : p \text{ is a quantifier-free type realised in }\M},
    \] 
    where we interpret the new relation symbols in the natural way, that is:
    \[
    \widehat{\M}_\mathsf{qfi}\vDash R_p(\bar{a}) \text{ if, and only if } \M\vDash p(\bar{a}),
    \]
    for every $\bar{a}\in \M$ and every quantifier-free type $p$ that is realised in $\M$. 
\end{definition}

An obvious remark about the quantifier-free type Morleyisation (and more generally about the $\L_{\infty,0}$-Morleyisation), which will be relevant later, is that equality of quantifier-free types does not change when one moves from $\M$ to $\widehat{\M}_\mathsf{qfi}$ (or $\widehat{\M}$). More precisely:

\begin{remark}\label{qf-type-Morleyzation-same-types}
    Let $\M$ be an $\L$-structure. Then, for all tuples $\bar{a},\bar{b}$ from $\M$ we have that:
    \[
        \qftp_{\widehat{\M}_\mathsf{qfi}}(\bar{a})=\qftp_{\widehat{\M}_\mathsf{qfi}}(\bar{b})\text{ if, and only if, }\qftp_{\M}(\bar{a})=\qftp_{\M}(\bar{b}).
    \]
\end{remark}
 
\subsection{Quantifier-Free type Isolators} The conclusion of \Cref{qf-type-Morleyzation-same-types} is a crucial condition for our applications. It is often, though, convenient to work with relational structures, rather than arbitrary ones. To this end, we consider the smallest relational reduct of the quantifier-free type Morleyisation of an $\L$-structure. More precisely:

\begin{definition}\label{def:isolator}
    Let $\M$ be an $\L$-structure. The \emph{quantifier-free type isolator}, or the \emph{$\L_{\infty,0}$-isolator} of $\M$ is the reduct, $\M_{\mathsf{iso}}$, of $\widehat{\M}_\mathsf{qfi}$ to the language $\L_\mathsf{iso} = \widehat{\L}_\mathsf{qfi}\setminus\L$, where $\widehat{\L}_\mathsf{qfi}$ is the language of $\widehat{\M}_\mathsf{qfi}$.
\end{definition}

Explicitly, in the notation above, $\L_{\mathsf{iso}}$ is such that for every $\bar{a}\in\M$, if $p = \qftp_{\M}(\bar{a})$, then there is a unique relation symbol $R_p\in \L_\mathsf{iso}$ such that for every $\bar{b}\in \M$:
\[
    \M_{\mathsf{iso}}\vDash R_p(\bar{b})\text{ if, and only if, } \qftp_{\M}(\bar{b}) = p.
\]

Of course, as commented on earlier, it is easy to see that the analogue of \Cref{qf-type-Morleyzation-same-types} still holds for quantifier-free type isolators. We write it down explicitly, for bookkeeping purposes:

\begin{remark}\label{qf-isolator-same-types}
    Let $\M$ be an $\L$-structure. Then, for all tuples $\bar{a},\bar{b}$ from $\M$ we have that:
    \[
    \qftp_{\M_\mathsf{iso}}(\bar{a})=\qftp_{\M_\mathsf{iso}}(\bar{b})\text{ if, and only if, }\qftp_{\M}(\bar{a})=\qftp_{\M}(\bar{b}).
    \]
\end{remark}

In the next sections, we will utilise this machinery to prove stronger versions of Scow's theorem, as discussed in the introduction. We will also make use of the quantifier-free type isolator construction, to prove a variant of the Kechris-Pestov-Todorcevic correspondence, as discussed in \Cref{sec:introduction}.
 
\section{Background: Ramsey Classes}\label{sec:ramsey classes}

In this \namecref{sec:ramsey classes} we recall the basics of structural Ramsey theory. Given $\L$-structures $A,B$ we write $\binom{A}{B}$ for the set of all embeddings of $B$ into $A$. Sometimes this notation is used to denote the set of all isomorphic copies of $B$ in $A$. Of course, in general, these two sets need not be equal. Nevertheless, we will soon show that in our context (i.e., that of classes with the \emph{Embedding Ramsey Property}, see \Cref{def:hp-jep-ap-erp}(4)), all structures are \emph{rigid} justifying our choice of terminology. (Recall: An $\L$-structure $\M$ is called \emph{rigid} if it has no non-trivial automorphisms.) 

\begin{definition}[Structural Erd\H{o}s-Rado Partition Arrow]
	Let $A,B,C$ be $\L$-structures (and without loss of generality assume that $A\subseteq B\subseteq C$) and let $r\in\N$. We write:
	\[
	{\arrows CB}^A_r
	\] 
	if for each colouring $\chi:\binom{C}{A}\to r = \Set{0,\dots,r-1}$ there exists some $\tilde B\in\binom{C}{B}$ such that $\chi\restr{\binom{\tilde B}{A}}$ is constant.
\end{definition}

\begin{definition}[HP, JEP, AP, ERP]\label{def:hp-jep-ap-erp}
    Let $\L$ be a first-order language and $\C$ a class of $\L$-structures. We say that $\C$ has the:
	\begin{enumerate}[(1)]
		\item \emph{Hereditary Property} (HP) if whenever $A\in\C$ and $B\subseteq A$ we have that $B\in\C$.
		\item \emph{Joint Embedding Property} (JEP) if whenever $A,B\in\C$ there is some $C\in\C$ such that both $A$ and $B$ are embeddable in $C$.
		\item \emph{Amalgamation Property} (AP) if whenever $A,B,C\in\C$ are such that $A$ embeds into $B$ via $e:A\into B$ and into $C$ via $f:A\into C$ there exist a $D\in\C$ and embeddings $g:B\into D$, $h:C\into D$ such that $g\circ e = h\circ f$.
		\item \emph{Embedding Ramsey Property} (ERP) if whenever $A,B\in\C$ are such that $A\subseteq B$, then there is some $C\in\C$ such that ${\arrows C B}^A_2$.
	\end{enumerate}
\end{definition}

\begin{remark}\label{erp all colours}
By a standard induction argument, if $\C$ has ERP then for all $r\in\N$ and all $A,B\in\C$ such that $A\subseteq B$, there is some $C\in\C$ such that ${\arrows C B}^A_r$.
\end{remark}

\begin{definition}[Ramsey Class]\label{def:rc}
    We say that an isomorphism-closed class of finite $\L$-structures, $\C$, is a \emph{Ramsey Class} if it has HP, JEP, and ERP.
\end{definition}

In \Cref{sec:local-finiteness} we generalise \Cref{def:rc} to classes of not necessarily finite structures.

\begin{example}
    In this updated terminology, Ramsey's theorem says that the class of all finite linear orders is a Ramsey class. 
\end{example}

We recall the following theorem of  Ne\v{s}et\v{r}il.
\begin{fact}[{\cite[Theorem 4.2(i)]{Nesetril_2005}}]\label{RC implies AP}
    Let $\C$ be a Ramsey class. Then $\C$ has AP.
\end{fact}

A natural generalisation of ERP comes from the notion of \emph{Ramsey degrees}, which roughly says that even if we cannot always find monochromatic copies of our structures, we have some ``global'' control over the number of colours that a structure can take. More precisely:

\begin{definition}
    Let $\C$ be a class of $\L$-structures and $A\in\C$. Given $d\in\N$, we say that $A$ has \emph{Ramsey degree $d$ (in $\C$)} if $d$ is the least positive integer such that for any $B\in\C$ with $A\subseteq B$ there is some $C\in\C$ such that $B\subseteq C$ and for any $r\in\N$ and any colouring $\chi: \binom{C}{A}\to r$ there exists some $\tilde B\in\binom{C}{B}$ such that $|\Im(\chi\restr{\binom{\tilde B}{A}})| \leq d$.  
\end{definition}

We note here that often, in literature, this notion is referred to sometimes as the \emph{small} Ramsey degree, to distinguish between so-called \emph{big} Ramsey degrees, see, for instance, \cite{Masulovic_2021} for a relevant discussion.

\begin{remark}\label{rmk:erp-iff-rd-is-1}
    In these terms, $\C$ has ERP if, and only if, for all $A\in\C$ we have that $A$ has Ramsey degree $1$.
\end{remark}

\begin{fact}[{\cite[Lemma~$2.9$, and Corollary~$2.10$]{Bodirsky_2015}}]\label{ramsey classes have rigid elts}
    Let $\C$ be a class of finite structures. Then, for all $A\in\C$, the Ramsey degree of $A$ is at least $|\Aut(A)|$. In particular, if $\C$ has ERP, then all members of $C$ are rigid.
\end{fact}

The following lemma about Ramsey classes will become useful later on:

\begin{lemma}\label{Ramsey-many-degs}
    Let $\C$ be a class of $\L$-structures. Let $A_1,\dots A_n\in \C$, be structures contained in some $B\in\C$, with $A_i$ having Ramsey degree $d_i$, for all $i\leq n$. Then, there is some $C\in\C$ such that $B\subseteq C$ and for all $r_i\in\N$ and all sets of colourings $\Set{\chi_i:\binom{C}{A_i}\to r_i:i\leq n}$, there is some $\tilde B\in\binom{C}{B}$ such that $|\Im(\chi_i\restr{\binom{\tilde B}{A_i}})|\leq d_i$, for all $i\leq n$.
\end{lemma}
\begin{proof}
    The proof is by induction on $n$. The base case is precisely the definition of $A_1$ having Ramsey degree $d_1$ in $\C$. For the inductive step, assume that we are given $A_1,\dots,A_{n+1},B\in\C$ such that $A_i\subseteq B$ and the Ramsey degree of $A_i$ is $d_i$, for all $i\leq n+1$. By the inductive hypothesis, there is some $C_0\in\C$ such that for all $r_i\in\N$, $i\leq n$, and all sets of colourings $\Set{\chi_i:\binom{C_0}{A_i}\to r_i:i\leq n}$, there is some $\tilde B\in\binom{C_0}{B}$ such that $|\Im(\chi_i\restr{\binom{\tilde B}{A_i}})|\leq d_i$, for all $i\leq n$. 
    
    Now, since $A_{n+1}$ has Ramsey degree $d_{n+1}$ in $\C$, there is some $C\in\C$ such that for all $r\in\N$ and all colourings $\chi: \binom{C}{C_0}\to r$ there exists some $\tilde C_0\in\binom{C}{C_0}$ such that $|\Im(\chi\restr{\binom{\tilde C_0}{A_{n+1}}})| \leq d_{n+1}$.
    
    We claim that this $C$ is the required structure. Indeed, suppose that we are given $r_i\in\N$, $i\leq n+1$, and a set of colourings $\Set{\chi_i:\binom{C}{A_i}\to r_i:i\leq n+1}$. Let $\tilde C_0\in\binom{C}{C_0}$ be such that $|\Im(\chi\restr{\binom{\tilde C_0}{A_{n+1}}})| \leq d_{n+1}$. Then by induction, there is some $\tilde B\in\binom{\tilde C_0}{B}$ such that $|\Im(\chi_i\restr{\binom{\tilde B}{A_i}})|\leq d_i$, for all $i\leq n$. But since $\binom{\tilde B}{A_{n+1}}\subseteq\binom{\tilde C_0}{A_{n+1}}$, we must have that:
    \[
    \left|\Im\left(\chi_{n+1}\restr{\binom{\tilde B}{A_{n+1}}}\right)\right|\leq \left|\Im\left(\chi\restr{\binom{\tilde C_0}{A_{n+1}}}\right)\right| \leq d_{n+1},
    \] 
    as required.
\end{proof}
As an immediate corollary of \Cref{rmk:erp-iff-rd-is-1} and \Cref{Ramsey-many-degs} we obtain the following useful fact:

\begin{corollary}\label{Ramsey many}
    Let $\C$ be a Ramsey class, and $A_1,\dots A_n\in \C$, be structures contained in some $B\in\C$. Then, there is some $C\in\C$ such that for all sets of colourings $\Set{\chi_i:\binom{C}{A_i}\to 2:i\leq n}$, there is some $\tilde B\in\binom{C}{B}$ such that $\chi_i\restr{\binom{\tilde B}{A_i}}$ is constant for all $i\leq n$.
\end{corollary}

\begin{definition}[Fraïssé Class]
We say that a countable class $\C$ of isomorphism types of finite $\L$-structures is a \emph{Fraïssé Class} if it has HP, JEP, and AP.
\end{definition}

Recall that for a structure $\M$ we define the \emph{age} of $\M$, denoted by $\Age(\M)$, to be the class of isomorphism types of finitely generated substructures of $\M$. We say that $\M$ is \emph{ultrahomogeneous} if every isomorphism between finitely generated substructures of $\M$ extends to an automorphism of $\M$. Given an $\L$-structure $\M$ and a subset $A\subseteq M$, we write $\langle A\rangle_\M$ for the substructure of $\M$ \emph{generated by $A$}. We say that $\M$ is \emph{locally finite} if for any finite $A\subseteq M$ we have that $\langle A\rangle_\M$ is finite.

The following celebrated theorem connects most of the notions defined above:

\begin{fact}[Fraïssé's Theorem]\label{fraisse}
    Let $\C$ be a non-empty Fraïssé class. Then there exists a unique, up to isomorphism, countable ultrahomogeneous $\L$-structure $\M$ such that $\Age(\M) = \C$.
\end{fact}

\begin{notation*}
    In the notation of \Cref{fraisse}, we will write $\Flim(\C)$ to denote the unique, up to isomorphism, countable ultrahomogeneous $\M$ such that $\Age(\M) = \C$.
\end{notation*}

Before we proceed, let us recall the central theorem of Kechris-Pestov-Todorcevic, from \cite{KechrisPestovTodorcevic_2005}, which we mentioned in the introduction. This theorem relates the topological dynamics of the automorphism group of an $\L$-structure $\M$ with ERP (in $\Age(\M)$). More precisely:

\begin{fact}[{\cite[Theorem 4.8]{KechrisPestovTodorcevic_2005}}, {\cite[Theorem 11.2.2]{Bodirsky_2021}}]\label{thm:kpt}
    Let $\C$ be a Fraïssé class, $\M=\Flim(\C)$, and $G=\Aut(\M)$. Then, the following are equivalent:
    \begin{enumerate}[(1)]
        \item The group $G$ is extremely amenable.
        \item $\C$ consists of rigid structures and has ERP.
        \item $\C$ is a Ramsey class and $\M$ admits an $\Aut(\M)$-invariant linear order, i.e., there exists a linear order $\preceq$ on $M$ such that for all $i,j\in M$ and all $\sigma\in\Aut(\M)$ we have that $i\preceq j$ if, and only if $\sigma(i)\preceq\sigma(j)$.
    \end{enumerate}
\end{fact}

Topological dynamics is not a central theme in our paper, although we make some minor remarks on the topology of the automorphism group of ``tame'' infinitary expansions of $\L$-structures. We direct the reader to \cite[Section~$1$]{KechrisPestovTodorcevic_2005}, for the required background.

\section{Background: Generalised Indiscernibles}\label{sec:gen inds}
Throughout this \namecref{sec:gen inds}, fix a first-order language $\L$ and let $T$ be a fixed complete $\L$-theory with infinite models. Let $\Mbb$ be a $\kappa(\Mbb)$-saturated and $\kappa(\Mbb)$-homogeneous model of $T$, for some large cardinal $\kappa(\Mbb)$ (for instance, a monster model). Throughout this section, we will always be working inside $\Mbb$. By a \emph{small} subset of $\Mbb$ we mean a subset $A\subseteq\Mbb$ such that $|A|<\kappa(\Mbb)$.

\begin{note}
In this \namecref{sec:gen inds}, unless otherwise stated, all sequences will consist of \emph{same-length} tuples.
\end{note}

\begin{definition}[Generalised Indiscernibles]
    Let $\L^\prime$ be a first-order language and $\Ncal$ an $\L^\prime$-structure. Given an $\Ncal$-indexed sequence of tuples $\mathcal{I} = (\bar a_i :i\in\Ncal)$ from $\Mbb$, and a small subset $A\subseteq\Mbb$, we say that $\mathcal{I}$ is an \emph{$\Ncal$-indexed indiscernible sequence over $A$} or simply that $\mathcal{I}$ is \emph{$\Ncal$-indiscernible over $A$} if for all $n\in\omega$ and all sequences $i_1,\dots,i_n,j_1,\dots,j_n$ from $\Ncal$ we have that:
    \[
    \begin{aligned}
        \text{If } \qftp_\Ncal(i_1,\dots,i_n) = \qftp_\Ncal^{\L^\prime}&(i_j,\dots,j_n) \\ \text{ then } &\tp_\Mbb(\bar a_{i_1},\dots,\bar a_{i_n}/A) =  \tp_\Mbb(\bar a_{j_1},\dots,\bar a_{j_n}/A).
    \end{aligned}
    \]
    If $A=\emptyset$, we say that $\mathcal{I}$ is simply an \emph{$\Ncal$-indiscernible sequence}.
\end{definition}

\begin{note}
    Henceforth, as is usual, we will not speak of (generalised) indiscernibility over a small set $A\subseteq\Mbb$. Instead, when relevant, we will add constants naming $A$ to $\L$ reducing this to indiscernibility over $\emptyset$ (in the $\L(A)$-structure $\Mbb$, where the new constant symbols are naming the elements of $A$).
\end{note}

\begin{example}
    As mentioned in the introduction, if $\Ncal = (\N,<)$ (or, more generally, $\Ncal$ is any infinite linear order), then $\Ncal$ -indiscernible sequences are simply \emph{indiscernible sequences}, also referred to as \emph{order-indiscernibles}. If $\Ncal = \N$ (or, more generally, $\Ncal$ is any infinite set in the language of pure equality), then $\Ncal$-indiscernible sequences are simply \emph{totally indiscernible sequences} also referred to as \emph{indiscernible sets}.
\end{example}

\begin{definition}[The Modelling Property]\label{def:mp}
    Let $\L^\prime$ be a first-order language, $\Ncal$ be an $\L^\prime$-structure, and $\mathcal{I} = (\bar a_i :i\in\Ncal)$ be an $\Ncal$-indexed sequence of tuples from $\Mbb$.
    
    \begin{enumerate}[(1)]
        \item Given an $\Ncal$-indexed sequence of tuples $\mathcal{J} = (\bar b_i : i\in\Ncal)$ from $\Mbb$, we say that $\mathcal{J}$ is \emph{(locally) based on} $\mathcal{I}$ if for all finite sets of $\L$-formulas $\Delta\subseteq\L$, all $n\in\N$ and all $i_1,\dots,i_n$ from $\Ncal$ there is some $j_1,\dots,j_n$ from $\Ncal$ such that:
			
        \begin{enumerate}[(i)]
            
			\item $\qftp_\Ncal^{\L^\prime}(i_1,\dots,i_n) = \qftp_\Ncal^{\L^\prime}(j_1,\dots,j_n)$
            
			\item $\tp^\Delta_\Mbb(\bar b_{i_1},\dots,\bar b_{i_{n}}) = \tp^\Delta_\Mbb(\bar a_{j_1},\dots,\bar a_{j_{n}})$ 
		\end{enumerate}
        
        \item We say that $\Ncal$-indiscernible sequences have the \emph{modelling property in $T$} if for each $\Ncal$-indexed sequence $\mathcal{I} = (\bar a_i:i\in\Ncal)$ of tuples from $\Mbb\vDash T$ there exists an $\Ncal$-indiscernible sequence $\mathcal{J}$ based on $\mathcal{I}$.
        
        \item We say that $\Ncal$-indiscernible sequences have the \emph{modelling property} if for all first-order theories $T$, they have the modelling property in $T$.
    \end{enumerate}
\end{definition}

The following is immediate, from the definition of the modelling property:

\begin{lemma}\label{same-qf-types-MP}
    Let $\cM_1, \cM_2$ be structures in languages $\cL_1,\cL_2$, respectively, with the same domain $M$. Assume 
    $\qftp_{\cM_1}(\bar{a})=\qftp_{\cM_1}(\bar{b})$ if, and only if, $\qftp_{\cM_2}(\bar{a})=\qftp_{\cM_2}(\bar{b})$
    for every $\bar{a},\bar{b}\in M$. Then $\cM_1$ has MP if, and only if, $\cM_2$ has MP.
\end{lemma}

We will simply say that $\Ncal$ \emph{has the modelling property} (MP) as an abbreviation for ``$\Ncal$-indiscernible sequences have the modelling property''.

Before we state Scow's theorem, we recall the following definition, from \cite{Scow_2015}:

\begin{definition}[QFI, {\cite[Definition 2.3]{Scow_2015}}]
    Let $\Ncal$ be an $\L^\prime$-structure. We say that $\Ncal$ is \emph{QFI} if for any quantifier-free type $q(x)$ realised in $\Ncal$ there is a quantifier-free formula $\theta_q(x)$ such that $\Th_{\forall}(\Ncal)\cup\Set{\theta_q(x)}$ is consistent and: 
    \[
    \Th_\forall(\Ncal)\cup\Set{\theta_q(x)}\vdash q(x).
    \]
\end{definition}

As commented in \cite{Scow_2015}, the term QFI refers to quantifier-free types (in particular, realised quantifier-free types) being isolated (and, in fact, by a quantifier-free formula). The main theorem of \cite{Scow_2015}, which draws a deep connection between generalised indiscernibles and Ramsey classes, is the following:

\begin{fact}[{\cite[Theorem 3.12]{Scow_2015}}]\label{thm:scow}
    Fix a first-order language $\L^\prime$ and assume that $\L^\prime$ contains a distinguished binary relation symbol $\leq$. Let $\C$ be a class of finite $\L^\prime$-structures and $\Ncal$ an infinite, locally finite $\L^\prime$-structure with QFI such that $\Age(\Ncal) = \C$ and $(\Ncal,\leq)\vDash\mathsf{LO}$ (where $\mathsf{LO}$ expresses that $\leq$ is a linear order). Then, the following are equivalent:
    \begin{enumerate}[(1)]
        \item \label{scow rc} $\C$ is a Ramsey class.
        \item \label{scow mp} $\Ncal$-indiscernibles have the modelling property.
    \end{enumerate}
\end{fact}

\section{Existence of indiscernible sequences}\label{sec:existence}

In this \namecref{sec:existence}, we discuss necessary and sufficient conditions on an $\L'$-structure $\Ncal$ for $\Ncal$-indiscernibles to exist in (sufficiently saturated) models of arbitrary theories. The key definition of this \namecref{sec:existence} is the following:

\begin{definition}\label{def:qford}
Let $\C$ be a class of $\L'$-structures. We say that $\C$ is \emph{$\L'_{\infty,0}$-orderable} if there is a (possibly infinite) Boolean combination of atomic and negated atomic $\L'$-formulas $\Phi(x,y):=\bigvee_{i\in I}\bigwedge_{j\in J_i} \vphi_{j_i}^{(-1)^{n_{j_i}}}(x,y)$, such that $\Phi$ is a linear order for every structure in $\C$, i.e. if there is an $\L_{\infty,0}$-formula which is a linear order on all structures in $\C$. We say that an $\L'$-structure $\Ncal'$ is \emph{$\L'_{\infty,0}$-orderable} if $\Age(\M)$ is $\L'_{\infty,0}$-orderable.
\end{definition}

We start with the following easy \namecref{qftp orders}, which is really an immediate consequence of \Cref{same-qf-types-MP}:

\begin{lemma}\label{qftp orders}
    Let $\Ncal$ be an $\L^\prime$-structure and let $\preceq$ be a linear order on $\Ncal$ which is a union of quantifier-free types. Let $\Ncal_o = (\Ncal,\preceq)$. If $\Ncal_o$-indexed indexed indiscernibles have the modelling property, then $\Ncal$-indiscernibles have the modelling property.
\end{lemma}

We observe that given an $\L^\prime$-structure $\Ncal$ the property that $\Age(\Ncal)$ is $\L^\prime_{\infty,0}$-orderable is the minimal requirement for $\Ncal$-indiscernibles to exist in every theory. We start by considering universal theories only:

\begin{lemma}\label{existence2 of inds equivalent to qf-definable structure}
    Given an $\L^\prime$-structure $\Ncal$ and a universal $\cL$-theory $T$, the following are equivalent:
    \begin{enumerate}[(1)]
        \item\label{existence2 items:inds} There is some $\cM\models T$ and a non-constant $\cN$-indiscernible sequence of singletons in $\cM$.
        \item\label{existence2 items:model} There is some $\cM\models T$ such that:
        \begin{itemize}
            \item $N\subseteq M$ and $\langle N\rangle = \M$, where $\langle N\rangle$ denotes the substructure of $\M$ generated by $N$.
            \item Every formula in $\L$ on the elements of $N$ is given by an $\L^\prime_{\infty,0}$-formula.
        \end{itemize}
    \end{enumerate}
\end{lemma}

\begin{remark}
    In case $\L$ is relational, the \namecref{existence2 of inds equivalent to qf-definable structure} actually gives us that there is some \emph{$\L^\prime_{\infty, 0}$-definable model} of $T$ on $\Ncal$, where, by \emph{$\L^\prime_{\infty, 0}$-definable structure on $\cN$}, we mean an $\cL$-structure $\cM$ with the same domain as $\cN$, such that each relation in $\cL$ on $\cM$ is given by an $\cL'_{\infty,0}$-formula in $\cN$.
\end{remark}

\begin{proof}
    For the implication (\labelcref{existence2 items:model})$\implies$(\labelcref{existence2 items:inds}), observe that if $T$ is a universal theory and $\M\vDash T$ is an $\L^\prime_{\infty,0}$-definable structure on $\Ncal$, as in (\labelcref{existence2 items:model}), then $N$ is, by definition, an $\Ncal$-indiscernible sequence of singletons in $\cM$, where each element is indexed by itself.
            
    Conversely, for the implication (\labelcref{existence2 items:inds})$\implies$(\labelcref{existence2 items:model}), suppose that $\M\vDash T$ and $(a_i: i\in \Ncal)$ are as in (\labelcref{existence2 items:inds}). Then, since $T$ is universal and $\Set{a_i: i\in \Ncal}\subseteq\M$, it follows that $\braket{\Set{a_i: i\in \Ncal}}_\Mcal\models T$. Let $\cM'$ be the $\L$-structure induced on $\braket{\Set{a_i: i\in \Ncal}}$ by the bijection $i\mapsto a_i$. Identifying $a_i$ with $i$, we assume $N\subseteq M$.
    To prove (\labelcref{existence2 items:model}), we show that every $\lL$-formula on $N$ is given by an $\lL'_{\infty,0}$-formula. To this end, let 
    $\vphi(x_1,\dots,x_n)$ be an $\lL$-formula, and let $\Psi:=\Set{\qftp_{\cN}(i_1,\dots, i_n) | \cM\models \vphi(a_{i_1},\dots, a_{i_n})  }$. 
    By definition of $\Psi$, if $\cM\models \vphi(a_{i_1},\dots, a_{i_n})$, then $i_1,\dots,i_n\models p$ for some
    $p\in \Psi$. On the other hand, by $\cN$-indiscernibility of $(a_i : i\in \cN)$, for every $p\in \Psi$, if $i_1,\dots, i_n\models p$, then $\cM\models \vphi(a_{i_1},\dots, a_{i_n})$. Thus,
    \[ 
        \cM\models \vphi(a_{i_1},\dots, a_{i_n}) \text{ if, and only if, } \cN\models \bigvee_{p\in \Psi}\left(\bigwedge_{\psi\in p} \psi(i_1,\dots,i_n)\right), 
    \]
    witnessing (\labelcref{existence2 items:model}).
\end{proof}

\begin{proposition}\label{existence of inds}
    Given an $\L^\prime$-structure $\Ncal$, the following are equivalent:
    \begin{enumerate}[(1)]
        \item\label{existence items:age orderable} $\Age(\Ncal)$ is $\L^\prime_{\infty,0}$-orderable.
        \item\label{existence items:inds exist in LO} There is an infinite $\Set{\leq}$-structure $\M$ such that $(\M,\leq)\vDash \mathsf{LO}$ and a non-constant $\Ncal$-indiscernible sequence in $\M$.
        \item\label{existence items:singleton inds exist in LO} There is an infinite $\Set{\leq}$-structure $\M$ such that $(\M,\leq)\vDash \mathsf{LO}$ and a non-constant $\Ncal$-indiscernible sequence \emph{of singletons} in $\M$.
        \item\label{existence items:inds exist in every theory} For every $\L$-theory $T$ with infinite models and every sufficiently saturated model $\Mbb\vDash T$ there is a non-constant $\Ncal$-indiscernible sequence in $\Mbb$.
        \item\label{existence items:singleton inds exist in every theory} For every $\L$-theory $T$ with infinite models, and every sufficiently saturated model $\Mbb\vDash T$ there is a non-constant $\Ncal$-indiscernible sequence \emph{of singletons} in $\Mbb$.
        \item\label{existence items:inds exist in every univ theory} For every \emph{universal} $\L$-theory $T$ with infinite models and every sufficiently saturated model $\Mbb\vDash T$ there is a non-constant $\Ncal$-indiscernible sequence in $\Mbb$.
        \item\label{existence items:singleton inds exist in every univ theory} For every \emph{universal} $\L$-theory $T$ with infinite models, and every sufficiently saturated model $\Mbb\vDash T$ there is a non-constant $\Ncal$-indiscernible sequence \emph{of singletons} in $\Mbb$.
        \item\label{existence items:age universalable} For every relational universal $\L$-theory $T$ with infinite models, there is some $\L^\prime_{\infty, 0}$-definable model of $T$ on $\Ncal$.
    \end{enumerate}
    
    \begin{proof}
        Notice that $\LO$ is a universal theory, so (\labelcref{existence items:age universalable})$\implies$(\labelcref{existence items:age orderable}), (\labelcref{existence items:inds exist in every univ theory})$\implies$(\labelcref{existence items:inds exist in LO}), and (\labelcref{existence items:singleton inds exist in every univ theory})$\implies$(\labelcref{existence items:singleton inds exist in LO}). 
        Since order-indiscernible sequences exist in every theory with infinite models, and if $\cN$ is an ordered structure, then $(\cN\upharpoonright \Set{<})$-indiscernibility implies $\cN$-indiscernibility, we have (\labelcref{existence items:age orderable})$\implies$(\labelcref{existence items:inds exist in every theory}). 
        The implications (\labelcref{existence items:inds exist in every theory})$\implies$(\labelcref{existence items:singleton inds exist in every theory}), (\labelcref{existence items:singleton inds exist in every theory})$\implies$(\labelcref{existence items:singleton inds exist in every univ theory}),
        (\labelcref{existence items:inds exist in every theory})$\implies$(\labelcref{existence items:inds exist in every univ theory}), (\labelcref{existence items:inds exist in every univ theory})$\implies$(\labelcref{existence items:singleton inds exist in every univ theory}), 
        and (\labelcref{existence items:inds exist in LO})$\implies$(\labelcref{existence items:singleton inds exist in LO}) are trivial.
        The equivalences (\labelcref{existence items:age orderable})$\iff$(\labelcref{existence items:singleton inds exist in LO}) and (\labelcref{existence items:age universalable})$\iff$(\labelcref{existence items:singleton inds exist in every univ theory}) are by \Cref{existence2 of inds equivalent to qf-definable structure}. We leave it to the reader to verify that the digraph of implications is strongly connected.
    \end{proof}
\end{proposition}

This proposition gives us the following, almost immediate corollary:

\begin{corollary}\label{mp-implies-ordered-expansion}
    Let $\Ncal$ be an $\L^\prime$-structure with MP. Then $\Ncal$ is $\L'_{\infty,0}$-orderable.
\end{corollary}

We will use this \namecref{mp-implies-ordered-expansion}, to show that the order assumption in \Cref{thm:scow} is not necessary. In particular, the key property that we obtain from the \namecref{mp-implies-ordered-expansion} is that if $\Ncal_o$ is the $\L'_{\infty,0}$-definable ordered expansion of $\Ncal$ then, for all tuples $\bar a,\bar b$ from $N$ we have that $\qftp_{\Ncal}(\bar{a})=\qftp_{\Ncal}(\bar{b})$ if, and only if, $\qftp_{\Ncal_o}(\bar{a})=\qftp_{\Ncal_o}(\bar{b})$.

\section{Removing Assumptions From Scow's Theorem}\label{sec:removing ass}
The goal of this section is to show that in \Cref{thm:scow}, the assumptions of (1) $\Ncal$ having QFI, and (2) $\Ncal$ being linearly ordered can be removed. We will first show that the QFI assumption can be removed (\Cref{sec:qfi ass}), by using \emph{quantifier-free type Morleyisations} (as introduced in \Cref{def:qf-type-Morleyisation}). Then, we will show that the order assumption can be removed in the case $\Age(\Ncal)$ is countable (\Cref{sec:reduce to countable}), and using this we will finally show that this can be done without any assumptions on the cardinality of $\Age(\Ncal)$. The main result of this section is \Cref{rc mp}.

\subsection{The QFI assumption}\label{sec:qfi ass}
Let us start this \namecref{sec:qfi ass} by pointing out some basic results regarding the quantifier-free type Morleyisation. The following remark is immediate from the definition of the quantifier-free type Morleyisation:

\begin{remark}\label{qf-type-Morleyzation-qfi}
    Let $\M$ be an $\L$-structure. Then $\widehat{\M}_\mathsf{qfi}$ has QFI.
\end{remark}

The next lemma is the analogue of \Cref{same-qf-types-MP}, for ERP:

\begin{lemma}\label{same-qf-types-ERP}
    Let $\cM_1, \cM_2$ be structures in languages $\cL_1,\cL_2$, respectively, with the same domain $M$. Assume 
    $\qftp_{\cM_1}(\bar{a})=\qftp_{\cM_1}(\bar{b})$ if, and only if, $\qftp_{\cM_2}(\bar{a})=\qftp_{\cM_2}(\bar{b})$
    for every $\bar{a},\bar{b}\in M$. Then $\cM_1$ has ERP if, and only if, $\cM_2$ has ERP.
\end{lemma}
\begin{proof}
    The \namecref{same-qf-types-ERP} follows from the definition of ERP and from the fact that for any two subsets $A,B\subseteq M$, and $i\in \Set{1,2}$, letting $\cA_i$ and $\cB_i$ be the substructures of $\cM_i$ induced on $A$ and $B$, respectively, $\binom{\cB_1}{\cA_1} = \binom{\cB_2}{\cA_2} $.
\end{proof}

It is almost immediate that the QFI assumption from \Cref{thm:scow} can be removed. More precisely:

\begin{corollary}\label{cor:scow-no-qfi}
     \Cref{thm:scow} holds without assuming that $\Ncal$ has QFI. 
\end{corollary}
\begin{proof}
    By \Cref{qf-type-Morleyzation-same-types} we have that $\Ncal$ and $\widehat{\Ncal}_\mathsf{qfi}$ satisfy the assumptions of \Cref{same-qf-types-MP,same-qf-types-ERP}. Thus, it suffices to show that the \namecref{cor:scow-no-qfi} holds for $\widehat{\Ncal}_\mathsf{qfi}$, but this is immediate from \Cref{thm:scow} and \Cref{qf-type-Morleyzation-qfi}.
\end{proof}

\subsection{Removing the order assumption}

We start by discussing the \emph{countable} case, that is, we prove \Cref{rc mp} when $\Age(\Ncal)$ is countable.

\begin{lemma}\label{countable rc order}
    Let $\Ncal$ be an $\L^\prime$-structure with the embedding Ramsey property and assume that $\Age(\Ncal)$ is countable. Then $\Ncal$ is $\L_{\infty,0}$-orderable.
\end{lemma}

\begin{proof}
    The result follows by combining \Cref{RC implies AP}, to obtain that $\C$ is a Fraïssé class, \Cref{fraisse} to obtain the ultrahomogeneous Fraïssé limit $\Ncal^\prime$ and \cite[Proposition 2.25]{Bodirsky_2015}, which says that ultrahomogeneous structures whose ages are Ramsey classes admit automorphism-invariant linear orders. In particular, we obtain an $\Aut(\Ncal^\prime)$-invariant order $\preceq$ on $\Ncal^\prime$.
    
    It is then clear, from the fact that $\Ncal^\prime$ is ultrahomogeneous that $\preceq$ is given by a (possibly infinite) Boolean combination of atomic or negated atomic $\L^\prime$-formulas, thus $\Age(\Ncal)$ is $\L^\prime_{\infty,0}$-orderable.
    
    Let $\Phi(x,y)$ be the $\L_{\infty,0}$-formula which defines $\preceq$ on $\Ncal^\prime$. We claim that $\Phi(x,y)$ defines a linear order on $\Ncal$, as well. By construction, $\Age(\Ncal) = \Age(\Ncal^\prime) = \C$, and $\Phi(x,y)$ defines a linear order on each $A\in\C$. Assume for a contradiction that $\Phi(x,y)$ does not define a linear order on $\Ncal$, then there is a finite substructure $A\subseteq\Ncal$ on which $\Phi(x,y)$ is not a linear order, but $A\in\C$, so, indeed, we have a contradiction.
\end{proof}

Combining the lemma above with \Cref{thm:scow} we obtain the following corollaries:

\begin{corollary}\label{countable: rc implies mp}
    Let $\L^\prime$ be a first-order language, $\Ncal$ an infinite locally finite $\L^\prime$-structure such that $\Age(\Ncal)$ is countable. If $\Ncal$ has the embedding Ramsey property, then it has the modelling property.
\end{corollary}

\begin{proof}
    Observe that if $\age(\cN)$ is a countable Ramsey class, then by \Cref{countable rc order}, there is an $\L'_{\infty,0}$-definable linear order $\preceq$ on $\Ncal$. By \Cref{same-qf-types-ERP}, we have that $(\Ncal,\preceq)$ has ERP since $\Ncal$ has ERP, so by \Cref{cor:scow-no-qfi} we have that $(\Ncal,\preceq)$ has MP. Finally, by \Cref{same-qf-types-MP}, $\cN$ has MP.
\end{proof}

\begin{corollary}\label{countable: mp implies rc}
    Let $\L^\prime$ be a first-order language and $\Ncal$ an infinite locally finite $\L^\prime$-structure. If $\Ncal$ has the modelling property, then $\Ncal$ has the embedding Ramsey property.
\end{corollary}

\begin{proof}
    By \Cref{mp-implies-ordered-expansion} we may also assume that $\Ncal$ is linearly ordered and still has MP, since the order is $\L'_{\infty,0}$-definable, by \Cref{same-qf-types-MP}. Finally, the result follows from \Cref{cor:scow-no-qfi}.
\end{proof}

\subsubsection{Reduction to countability} \label{sec:reduce to countable}

The main result of this \namecref{sec:reduce to countable} is the following proposition:

\begin{proposition}\label{MP for union of countable}
    Let $\cN$ be an $\L^\prime$-structure, and let $\Set{\cN_i | i\in I}$ be $\L^\prime$-structures, such that:
    \begin{enumerate}
        \item $\cN = \bigcup_{i\in I}\cN_i$.
        \item For every finite $A\subseteq \cN$, there is some $i\in I$, such that $A\subseteq \cN_i$.
    \end{enumerate}
    If for all $i\in I$ we have that $\cN_i$ has the modelling property, then $\cN$ has the modelling property.
\end{proposition}

We recall the following from \cite{Scow_2015}:

\begin{definition}[{\cite[Definition 2.10]{Scow_2015}}]\label{def:ind}
    Let $\L^\prime$ be a first-order language and fix an $\L^\prime$-structure $\Ncal$. We define $\ind(\Ncal, \cL)$ to
    be the following partial type in an $\Ncal$-indexed sequence of variables $(x_i:i\in\Ncal)$:
    \begin{align*}
        \ind(\Ncal,\cL)(x_i:i\in \Ncal):= 
        \large\{
        &\vphi(x_{i_1},\dots, x_{i_n})\limplies\vphi(x_{j_1},\dots, x_{j_n}) : n<\omega, \bar{i}, \bar{j}\in N^n, \\
        & \qftp_{\Ncal}\left(\bar{i}\right) = \qftp_{\Ncal}\left(\bar{j}\right),\, \vphi(x_1,\dots, x_n)\in \cL 
        \large\}.
    \end{align*}
\end{definition}

The following is a variant of {\cite[Definition 2.11]{Scow_2015}}.

\begin{definition}[Finitely satisfiable]\label{def:scow fin satisfiable}
    Let $\Gamma(x_i : i\in I)$ be a (not necessarily complete) $\L$-type, and let $\mathcal{I} = (a_i : i\in \Ncal)$ be an $\Ncal$-indexed set of parameters in $\bM$. We say that $\Gamma$ is \emph{finitely satisfiable} in $\mathcal{I}$ 
    if for every finite $\Gamma_0\subseteq \Gamma$ and for every finite $A\subseteq N$, there is a $B\subseteq N$, a bijection $f:A\to B$, and an enumeration $\bar{i}$ of $A$ such that:
    \[
    \qftp_{\Ncal}(\overline{i})=\qftp_{\Ncal}\left(f\left(\overline{i}\right)\right)\text{ and }(a_{f(i)} : i\in A)\models \Gamma_{0}\restr{\Set{x_i : i\in A}}.
    \]
\end{definition}

\begin{lemma}\label{MP iff ind fin satisfiable}
    Fix an $\L^\prime$-structure $\Ncal$, and a (not necessarily indiscernible) $\Ncal$-indexed sequence $\mathcal{I} = (a_i:i\in \Ncal)$. Then, the following are equivalent: 
    \begin{enumerate}[(1)]
        \item \label{MP iff ind fin satisfiable:ind fs} $\ind(\Ncal; \L)$ is finitely satisfiable in $\mathcal{I}$.
        \item \label{MP iff ind fin satisfiable: locally based}There is an $\Ncal$-indiscernible $\mathcal{J} = (b_i:i\in \Ncal)$ which is locally based on $\mathcal{I}$.
    \end{enumerate}
\end{lemma}

\begin{proof}
    The implication (\labelcref{MP iff ind fin satisfiable:ind fs})$\implies$(\labelcref{MP iff ind fin satisfiable: locally based}) follows as in \cite[{Proposition 2(6)}]{Scow_2015}, replacing $\Gamma$ with $\Gamma_0$ and using compactness.

    For the implication (\labelcref{MP iff ind fin satisfiable: locally based})$\implies$(\labelcref{MP iff ind fin satisfiable:ind fs}) we argue as follows. Pick arbitrary finite subsets $\Gamma_0\subseteq \ind(\Ncal,\cL)$ and $A\subseteq N$. Since $\mathcal{J}$ is an $\Ncal$-indiscernible based on $\mathcal{I}$, there is some $B\subseteq \Ncal$, a bijection $f:A\to B$, and an enumeration $\bar{i}$ of $A$ such that $\qftp_{\Ncal}(\bar{i})=\qftp_\Ncal\left(f\left(\bar{i}\right)\right)$ and $\tp^{\Gamma_0}_\Mbb(d_{i}:i\in A)=\tp_\Mbb^{\Gamma_0}(a_{f(i)}:i\in A)$. Also, $(b_i:i\in B)\models \Gamma_0\restr{\Set{x_{i}:i\in B}}$, by the (generalised) indiscernibility of $\mathcal{J}$, and hence $(a_{f(i)}:i\in A)\models \Gamma_0\restr{\Set{x_{i}:i\in A}}$.
\end{proof}

Now we are ready to prove \Cref{MP for union of countable}.

\begin{proof}[Proof of \Cref{MP for union of countable}]
    Let $\mathcal{I} = (a_i:i\in\Ncal)$ be an $\Ncal$-indexed sequence of tuples from $\Mbb$. We need to show that there exists an $\Ncal$-indiscernible sequence $\mathcal{J}$ which is locally based on $\mathcal{I}$. By \Cref{MP iff ind fin satisfiable}, it suffices to show that $\ind(\Ncal, \L)$ is finitely satisfiable in $\mathcal{I}$.

    To this end, let $\Gamma_0\subseteq\ind(\Ncal,\L)$ and $A\subseteq N$ be finite subsets. We need to find some $B\subseteq N$, and a bijection $f:A\to B$ such that:
    \[
    \qftp_{\Ncal}(\bar{j})=\qftp_{\Ncal}\left(f\left(\bar{j}\right)\right)\text{ and }(a_{f(j)} : j\in A)\models \Gamma_{0}|_{\Set{x_j : j\in A}},
    \]
    where $\bar j$ is some enumeration of $A$.

    By assumption, there is some substructure $\Ncal_i\subseteq\Ncal$ such that $A\subseteq \Ncal_i$, and since $\Ncal_i$-indiscernibles have the modelling property, we know that $\ind(\Ncal_i,\L)$ is finitely satisfiable in the subsequence
    $\mathcal{I}\restr{\Ncal_i} := (a_{i_k} : i_k\in\Ncal_i)$, which by assumption on $\Ncal_i$ contains the part of $\mathcal{I}$ indexed by $A$. This by definition, means that there is some $B\subseteq N_i$ and a bijection $f:A\to B$ such that:
    \[
    \qftp_{{\Ncal_i}}(\bar{j})=\qftp_{{\Ncal_i}}\left(f\left(\bar{j}\right)\right)\text{ and }(a_{f(j)} : j\in A)\models \Gamma_{0}|_{\Set{x_j : j\in A}},
    \]
    for an enumeration $\bar j$ of $A$.

    However, observe that $\Ncal_i$ is a substructure of $\Ncal$, so it follows that $\qftp_{\Ncal_i}(\bar a) = \qftp_{\Ncal}(\bar a)$, for all tuples $\bar a$ from $\Ncal_i$. Hence, the \namecref{MP for union of countable} follows, since the witnesses of the finite satisfiability of $\ind(\Ncal_i,\L)$, witness the finite satisfiability of $\ind(\Ncal,\L)$, as required.
\end{proof}

Now, we have all the required tools to show our main result. To summarise, the countable case of our main theorem follows, simply by observing that in \Cref{countable: mp implies rc}, the order is $\L_{\infty,0}$-definable, and hence it may be dropped, without this affecting the modelling property. Below, we prove the theorem for arbitrary $\L'$-structures.

\begin{theorem}\label{rc mp}
    Let $\L^\prime$ be a first-order language and $\Ncal$ an infinite locally finite $\L^\prime$-structure. Then, the following are equivalent:
    \begin{enumerate}[(1)]
        \item \label{main-thm-rc} $\Ncal$ has the embedding Ramsey property.
        \item \label{main-thm-mp}$\Ncal$ has the modelling property.
    \end{enumerate}
\end{theorem}
\begin{proof}
    The fact that (\labelcref{main-thm-mp})$\implies$(\labelcref{main-thm-rc}) was already proved in \Cref{countable: mp implies rc}. For the implication (\labelcref{main-thm-rc})$\implies$(\labelcref{main-thm-mp}), observe that since $\age(\Ncal)$ is a Ramsey class, we can find countable $\L^\prime$-structures $(\Ncal_i:i\in I)$, for some indexing set $I$, satisfying the conditions of \Cref{MP for union of countable}, i.e. such that:
    \begin{enumerate}
        \item $\Ncal = \bigcup_{i\in I}\Ncal_i$.
        \item For every finite $A\subseteq \Ncal$, there is some $i\in I$, such that $A\subseteq \Ncal_i$.
        \item For every $i\in I$ we have that $\Age(\Ncal_i)$ is a Ramsey class.
    \end{enumerate}
    To obtain the $\Ncal_i$, we proceed as follows. Let $(A_i:i\in I)$ be some enumeration (not necessarily countable) of all finite substructures of $\Ncal$. For each finite substructure $A_i$ we construct $\Ncal_i$ as follows. Let $B_0^i:= A_i$, and then take $B_1^i\in\Age(\Ncal)$ to be the $\L$-structure promised by \Cref{erp all colours}, with respect to all finite substructures of $B_0^i$. Repeat this process inductively, to obtain $(B_j^j:j\in\omega)$ and let $\Ncal_i = \bigcup_{j\in\omega} B_j^i$. By construction, $\Age(\Ncal_i)$ is still a Ramsey class, and it is easy to see that $\Ncal_i$ is countable, since it is the union of an increasing chain of finite structures. 
    Then, by \Cref{countable: rc implies mp}, for each $i\in I$, since $\Age(\Ncal_i)$ is a countable Ramsey class, $\cN_i$ has MP.  From \Cref{MP for union of countable} we can conclude that $\Ncal$ has MP. 
\end{proof}

As an immediate corollary of \Cref{rc mp} and \Cref{mp-implies-ordered-expansion} we obtain the following:

\begin{corollary}\label{rc qford}
    Let $\Ncal$ be an $\L$-structure with the embedding Ramsey property. Then $\Ncal$ is \qford.
\end{corollary}

\section{Around Local Finiteness}\label{sec:local-finiteness}

So far, we have only dealt with classes of finite $\L^\prime$-structures, or, to the same extent, classes of the form $\Age(\Ncal)$, where $\Ncal$ is a locally finite $\L^\prime$-structure. In this \namecref{sec:local-finiteness} we want to expand our discussion to deal with the embedding Ramsey property in a wider context of structures that are not necessarily locally finite. 

Let $\Ncal$ be an (arbitrary) $\L^\prime$-structure and let $\C = \Age(\Ncal)$. Recall that for a subset $A\subseteq \Ncal$ we write $\langle A \rangle_{\Ncal}$ for the substructure of $\Ncal$ \emph{generated by} $A$. Given an $\L^\prime$-structure $B$ and a finite subset $A \subseteq B$, we overload the notation $\binom{A}{B}$ to mean:
\[
    \binom{B}{A} := \{\tilde A\subseteq B : \qftp(A) = \qftp(\tilde A)\},
\]
where $B$ is an $\L^\prime$-structure, and $A\subseteq B$ a finite subset (but not necessarily a substructure) of $B$.
 
\begin{definition}[f-ERP]\label{def:f-ERP}
    Let $\C$ be a class of $\L^\prime$-structures. We say that $\C$ has the \emph{finitary Embedding Ramsey Property} (f-ERP) if for any $M\in\C$ and any finite subsets $A,B\subseteq M$ there are an $N\in\C$ and a finite subset $C\subseteq N$ such that $\arrows{C}{B}^A_2$. We say that an $\L^\prime$-structure $\M$ has the \emph{finitary Embedding Ramsey Property} (f-ERP) if $\Age(\M)$ the finitary Embedding Ramsey Property.
\end{definition}

The analogue of \Cref{def:rc}, for f-ERP is the following definition:

\begin{definition}\label{def:finitary-rc}
    Let $\C$ be a class of $\L^\prime$-structures. We say that $\C$ is a \emph{finitary Ramsey class} if $\C$ has HP, JEP, and f-ERP.
\end{definition}

\begin{remark}
If $\C$ is a class of finite structures, then $\C$ has f-ERP if, and only if $\C$ has ERP. In particular, for any locally finite $\L^\prime$-structure $\Ncal$ we have that $\Age(\Ncal)$ is a Ramsey class if, and only if it is a finitary Ramsey class.
\end{remark}

\begin{proposition}\label{f-ERP iff M arrows finite sets}
    Let $\cM$ be a structure. Then $\M$ has f-ERP if, and only if, for all finite subsets $A\subseteq B\subseteq \cM$ we have that $\arrows{\M}{B}^A_2$.
\end{proposition}
\begin{proof}
    The proof is a standard compactness argument.
\end{proof}

\begin{definition}[Embedding Local Finiteness]
    An $\L$-structure $\M$ has the \emph{embedding local finiteness property} (elf) if for any finitely generated $B\subseteq \M$ and finite $A\subseteq B$ we have that $\binom{B} {A}$ is finite. More generally, an isomorphism-closed class $\C$ of $\L$-structures with HP has the \emph{embedding local finiteness property} if for all $B\in\C$ and all finite $A\subseteq B$ we have that $\binom{B}{A}$ is finite. 
\end{definition}

\begin{remark}\label{embedding local fin iff embedding fin subset}
    An $\L$-structure $\M$ has elf if, and only if, for every finitely generated substructure $B\subseteq \M$ and every tuple $\bar{a}\in \cM^n$, there is some finite $B'\subset B$ such that $\binom{B'}{\bar{a}}=\binom{B}{\bar{a}}$. More generally, a class $\C$ of finitely generated $\L$-structures has elf if, and only if, for every $B\in\C$ and finite $A\subseteq B$ there is some finite subset $B^\prime\subseteq B$ such that $\binom{B'}{A}=\binom{B}{A}$. 
\end{remark}

We argue that f-ERP is the correct generalisation of ERP, to classes of not necessarily finite structures. Our main point is that the main results about ERP for classes of finite structures generalise, almost immediately, to arbitrary structures, if one replaces ERP with f-ERP. In particular, in this context, we have analogues of both the Kechris-Pestov-Todorcevic correspondence and of \Cref{rc mp}.

\begin{theorem}\label{thm:kpt-fERP}
    Let $\M$ be an ultrahomogeneous $\L$-structure. Then, the following are equivalent:
    \begin{enumerate}[(1)]
        \item $\Aut(\M)$ is extremely amenable.
        \item $\M$ has the finitary embedding Ramsey property.
    \end{enumerate}
\end{theorem}

This \namecref{thm:kpt-fERP} essentially follows from \cite{KechrisPestovTodorcevic_2005}, in the case where $\M$ is countable, and in the general case, it follows both from the results of \cite{KrupinskiPillay_2022}, and from combining the results of \cite{Bartosova_2013} with the proof of the main theorem of \cite{KechrisPestovTodorcevic_2005}. Nonetheless, we give an alternative, completely rigorous proof, which showcases the power of the techniques introduced in \Cref{sec:infinitary-logic}. First, note the following:

\begin{remark}\label{rmk:isolator} 
    Let $\M$ be an $\L$-structure. Recall that by $\M_{\mathsf{iso}}$ we denote its $\L_{\infty,0}$-isolator, introduced in \Cref{def:isolator}.
    \begin{enumerate}[(1)]
        \item \label{same aut group} Then $\Aut(\M)=\Aut(\M_{\mathsf{iso}})$. In particular, $\Aut(\M)$ is extremely amenable if, and only if, $\Aut(\M_{\mathsf{iso}})$ is extremely amenable.
        \item \label{rp-iff-ferp} $\M$ has the finitary embedding Ramsey property if, and only if $\M_{\mathsf{iso}}$ has embedding Ramsey property. This follows easily from \Cref{qf-isolator-same-types}
    \end{enumerate}
\end{remark}

\begin{proof}[Proof of \Cref{thm:kpt-fERP}]
    Given the two remarks above, the theorem follows from \Cref{thm:kpt}, applied to $\M_{\mathsf{iso}}$, which is relational and ultrahomogeneous, if, and only if, $\M$ is ultrahomogeneous.
\end{proof}

We can also use the notion of f-ERP we just introduced to remove \emph{all} assumptions from Scow's theorem. This gives us the following result, which is the most general version of correspondence between generalised indiscernibles and Ramsey classes:

\begin{theorem}\label{frc mp}
    Let $\L^\prime$ be a first-order language and $\Ncal$ an $\L^\prime$-structure. Then, the following are equivalent:
    \begin{enumerate}[(1)]
        \item \label{extension-thm-frc} $\Ncal$ has the finitary embedding Ramsey property.
        \item \label{extension-thm-mp}$ \Ncal$ has the modelling property.
    \end{enumerate}
\end{theorem}

\begin{proof}
    This easily follows from \Cref{same-qf-types-MP}, \Cref{qf-isolator-same-types}, and \Cref{rc mp}, applied to $\Ncal_{\mathsf{iso}}$.
\end{proof}

Before we conclude this \namecref{sec:local-finiteness}, we introduce a stronger form of \Cref{def:f-ERP}, in the following definition, where we weaken the assumption that the subsets we want to colour are finite and only assume that they are finitely generated. More precisely:

\begin{definition}[inf-ERP]\label{def:inf-ERP}
    Let $\C$ be a class of $\L^\prime$-structures. We say that $\C$ has the \emph{infinitary Embedding Ramsey Property} (inf-ERP) if for any $M\in\C$ and any finitely generated substructures $A\subseteq B\subseteq M$ there is an $N\in\C$ and a finitely generated subset $C\subseteq N$ such that $\arrows{C}{B}^A_2$. We say that an $\L^\prime$-structure $\M$ has the \emph{infinitary Embedding Ramsey Property} (inf-ERP) if $\Age(\M)$ the infinitary Embedding Ramsey Property.
\end{definition}

\begin{proposition}\label{inf-ERP iff M arrows finitely generated}
    Let $\cM$ be a structure. Then $\cM$ has the infinitary embedding Ramsey property if, and only if, for all finitely generated substructures $A\subseteq B\subseteq \cM$ we have that $\arrows{\M}{B}^A_2$.
\end{proposition}

\begin{proof}
    Clearly, inf-ERP implies $\arrows{\cM}{B}^A_2$ for all finitely generated substructures $A\subseteq B\subseteq \cM$.

    For the other direction, by f-ERP and \Cref{f-ERP iff M arrows finite sets}, $\arrows{\cM}{B}^A_2$ for all finite subsets $A\subseteq B\subseteq \cM$. By elf and \Cref{embedding local fin iff embedding fin subset}, $\arrows{\cM}{B}^A_2$ for all finitely generated substructures $A\subseteq B\subseteq \cM$.
\end{proof}

\begin{proposition}\label{prop:inf-ERP-iff-finite-in-fg}
    Let $\C$ be a class of $\L$-structures. Then, the following are equivalent:
    \begin{enumerate}
        \item \label{prop:inf-ERP}$\C$ has the infinitary embedding Ramsey property.
        \item \label{prop:weaker-inf-ERP} For any $M\in\C$, any finitely generated substructure $B\subseteq M$, and any finite $A\subseteq B$ there is an $N\in\C$ and a finitely generated subset $C\subseteq N$ such that $\arrows{C}{B}^A_2$.
    \end{enumerate}
\end{proposition}

First, we state the following lemma:

\begin{lemma}\label{lem:f-ERP-implies-same-qftp-different-gens}
    Let $\C$ be a finitary Ramsey class of $\L$-structures. Then, for all $M\in\C$ and all distinct tuples $\bar b_1,\bar b_2$ from $M$ we have that if $\qftp_M(\bar b_1) = \qftp_M(\bar b_2)$, then $\langle\bar b_1\rangle_M \neq \langle \bar b_2\rangle_M$.
\end{lemma}

\begin{proof}
Indeed, suppose that this is not the case. Since $\langle \bar b_1\rangle_M = \langle \bar b_2\rangle_M$, we have that $\bar b_2\in\langle\bar b_1\rangle_M$, so there is an $\L$-term $\mathsf{t}(\bar x)$ (where $|\bar x| = |\bar b_i|$) such that $\mathsf{t}(\bar b_1) = \bar b_2$. Now, note that by \Cref{frc mp} and \Cref{mp-implies-ordered-expansion} it follows that $\C$ is $\L_{\infty,0}$-orderable. Since the tuples we started with were distinct, without loss of generality, we may assume that $\bar b_1 < \mathsf{t}(\bar b_1)$, in the lexicographic order induced by the order given by $\L_{\infty,0}$-orderability. But since $\qftp(\bar b_1) = \qftp(\mathsf{t}(\bar b_1))$, it follows that for all $n\in\N$ we have that $\mathsf{t}^{n}(\bar b_1) <\mathsf{t}^{n+1}(\bar b_1)$, since for all $n\in\N$, we have that $\qftp(\mathsf{t}^n(\bar b_1)) = \qftp(\bar b_1)$.

We claim now, that for any $N\in\C$ and any finite $C\subseteq M$, we have that ${C}\not\rightarrow({\bar b_1{}^\frown\bar b_2})^{\bar b_1}_2$, that is, ${C}\not\rightarrow({\bar b_1{}^\frown\mathsf{t}(\bar b_1)})^{\bar b_1}_2$. Fix a set $\Set{\bar b^{1},\dots,\bar b^{m}}\subseteq\binom{C}{\bar b_1}$ such that for all $\bar b\in \binom{C}{\bar b_1}$ there exist unique $k\leq m$ and $n\in \Z$ such that $\bar b = \mathsf{t}^n(\bar b^{k})$.
Then, the colouring:
 \[
 \begin{aligned}
 	\chi:\binom{C}{\bar b_1} &\to 2 \\
 	 \mathsf{t}^n(\bar b^{k}) &\mapsto 
 			\begin{cases}
 				0 &\text{ if } n\equiv 0\pmod{2}, \\
 				1 &\text{ otherwise}.
 			\end{cases}
 \end{aligned}
 \]
 is well-defined and cannot have any $\bar b_1$-monochromatic $\bar b_1{}^\frown\bar b_2$, contradicting the fact that $\C$ has f-ERP.
\end{proof}

\begin{proof}[Proof of \Cref{prop:inf-ERP-iff-finite-in-fg}]
    The implication (\labelcref{prop:inf-ERP})$\implies$(\labelcref{prop:weaker-inf-ERP}) is immediate. In the other direction, observe that (\labelcref{prop:weaker-inf-ERP}) implies that $\C$ has f-ERP, but then, by \Cref{lem:f-ERP-implies-same-qftp-different-gens}, we must have that no two distinct tuples with the same quantifier-free types generate the same structure, and (\labelcref{prop:inf-ERP}) follows.
\end{proof}

In the remainder of this section, we discuss some connections between f-ERP and inf-ERP. We will assume that we are working with structures in a language $\L$ where $|\L|<\aleph_\omega$.

\begin{lemma}\label{lem:bounds-on-fg}
    Let $\cC$ be a class of finitely generated $\L$-structures which are not all finite. Then there is some $B\in \cC$ such that the cardinality of $B$ is maximal with respect to $\cC$.
\end{lemma}

\begin{proof}
    If not, then there is a chain of length $\omega$ of infinite finitely generated structures in $\cC$ that increases strictly in cardinality. The union of this chain will be a countably generated structure of cardinality at least $\aleph_\omega$, contradicting the assumption on $|\L|$.
\end{proof}

A similar argument as in \Cref{lem:bounds-on-fg} yields the following: 
\begin{lemma}\label{lem:bounds-on-fg for embeddings}
    Let $\cC$ be a class of finitely generated $\L$-structures without elf. Then there is some $B\in \cC$ and finite $A\subseteq B$, such that 
    \[\left|\binom{B}{A}\right|=\max\Set{\left|\binom{X}{Y}\right| : X\in \cC,\  Y\subseteq X \text{ finite}}.\]
\end{lemma}

\begin{proof}
    If not, then there is a chain of $\Set{B_i}_{i<\omega}$ of finitely generated structures in $\cC$ and a family of finite sets $\Set{A_i}_{i<\omega}$ such that $A_i\subseteq B_i$ and 
     $\aleph_0\leq \left|\binom{B_i}{A_i}\right|<\left|\binom{B_{i+1}}{A_{i+1}}\right|$ for all $i<\omega$.
    Let $B:=\bigcup_{i<\omega}B_i$ and let $|A_i|=n_i$. Then $B$ is countably generated and 
    $\left|\bigcup_{i<\omega}B^{n_i}\right|\geq \left|\bigcup_{i<\omega}\binom{B_i}{A_i}\right|\geq \aleph_{\omega}$. But if $B$ is infinite, then $|B^{n}|\leq|B|$ for every $n\in \bN$. So
    $\left|\bigcup_{i<\omega}B^{n_i}\right|\leq \sum_{i<\omega}|B^{n_i}|\leq \sum_{i<\omega}|B| = \aleph_0\cdot |B| = |B|$, so in conclusion, $|B|\geq \aleph_{\omega}$,
    contradicting the assumption on $|\L|$.
\end{proof}

\begin{proposition}\label{inf-ERP implies  uncountable or embedding local fin}
    Let $\C$ be an infinitary Ramsey class of finitely generated $\L$-structures. Then $\C$ has the embedding local finiteness property.
\end{proposition}

\begin{proof}
    Suppose for a contradiction that $\C$ is an infinitary Ramsey class that does not have elf. Let $B\in \cC$ and $A\subseteq B$ as promised from \Cref{lem:bounds-on-fg for embeddings}. Let $C\in \cC$.
    We will construct a colouring $\chi:\binom{C}{A}\to \Set{0,1}$ such that for all $\tilde B\in\binom{C}{B}$ we have that $|\Im(\chi\restr{\binom{\tilde B}{A}})|>1$.
    
    Let $\kappa:=\left|\binom{C}{B}\right|$ and let $\Set{e_i}_{i<\kappa}$ be an enumeration of $\binom{C}{B}$. Let $B'\subseteq B$ be finite such that $\Braket{B'} = B$. Clearly, there is an injective function $f:\binom{C}{B'}\to \binom{C}{B}$, so $\kappa\leq \left|\binom{C}{B'}\right|$.
    By choice of $A$ and $B$, we have that $\left|\binom{C}{B'}\right|\leq \left|\binom{B}{A}\right|$, in particular $\kappa\leq \left|\binom{B}{A}\right|$.

    We construct, transfinitely, a sequence of partial colourings $(\chi_{i}:\binom{C}{A}\to \Set{0,1})_{i<\kappa}$ as follows. Let $\chi_0$ be the empty colouring. Assume that $(\chi_j)_{j<i}$ have been chosen and that for every $j<i$ we have that $\mathsf{dom}(\chi_j)$ is of size less than $\kappa$. Then, we may choose $f_i, g_i\in \binom{B}{A}$ such that $e_i\circ f_i(A), e_i\circ g_i \notin \bigcup_{j<i}\mathsf{dom}(\chi_j)$ and $f_i\neq g_i$. Let $\chi_i(e_i\circ f_i)=0$ and $\chi_i(e_i\circ f_i)=1$
    
    Let $\chi:\binom{C}{A}\to \Set{0,1}$ be any colouring such that $\chi\supseteq \bigcup_{i<\kappa} \chi_i$. It remains to show that for all $e\in\binom{C}{B}$ we have that 
    $|\Im(\chi\restr{e\circ \binom{B}{A}}|>1$. Let $i<\kappa$. 
    Then $\chi(e_i\circ f_i) = \chi_i(e_i\circ f_i) = 0$ and $\chi(e_i\circ g_i) = \chi_i(e_i\circ g_i) = 1$.
    Since for every $e\in\binom{C}{B}$, there is some $i<\kappa$ such that $e=e_i$, the \namecref{inf-ERP implies  uncountable or embedding local fin} is proven.
\end{proof}

\begin{proposition}\label{inf-ERP follows from f-ERP+embedding local fin}
    Let $\C$ be a finitary Ramsey class with the embedding local finiteness property. Then $\C$ is an infinitary Ramsey class.
\end{proposition}

\begin{proof}
    Let $B\subseteq M\in\C$ be a finitely generated substructure, and let $A\subseteq B$ be a finite set. By \Cref{embedding local fin iff embedding fin subset}, there is some finite $B'\subseteq B$ such that $\binom{B'}{A}=\binom{B}{A}$, and since $B$ is finitely generated, we may assume that $B'$ generates $B$. 
    
    By f-ERP, there is some $N\in\C$ and some finite $C\subseteq N$ such that $\arrows{C}{B^\prime}^A_2$. Let $C^\prime  = \langle C \rangle_N$. We claim that $\arrows{C^{\prime}}{B}^A_2$. Indeed, by construction, given a colouring $\chi:\binom{C^\prime}{A}\to\Set{0,1}$ there is some $\tilde{B}^\prime\in\binom{C}{B^\prime}$ such that $|\Im(\chi\restr\binom{\tilde{B}^\prime}{A})| = 1$.  But then, since $\qftp(\tilde{B}^\prime) = \qftp(B^\prime)$, we have that $\binom{\tilde{B}^\prime}{A} =\binom{B}{A}$, and hence the result follows.
\end{proof}

\begin{theorem}\label{inf-ERP iff embedding local fin (assuming countable ERP)}
    Let $\C$ be an isomorphism-closed class of $\L$-structures with the hereditary property and the joint embedding property. Then, the following are equivalent:
    \begin{enumerate}
        \item $\C$ has the infinitary embedding Ramsey property.
        \item $\C$ has the finitary embedding Ramsey property and embedding local finiteness.
    \end{enumerate}
\end{theorem}

\begin{proof}
    The \namecref{inf-ERP iff embedding local fin (assuming countable ERP)} follows from \Cref{inf-ERP implies  uncountable or embedding local fin} and \Cref{inf-ERP follows from f-ERP+embedding local fin}.
\end{proof}

\section{Concluding Remarks}\label{sec:questions}

In \Cref{sec:local-finiteness}, we introduced two notions that extend the embedding Ramsey property to arbitrary classes (i.e., classes not necessarily consisting of finite structures), namely the \emph{finitary} and \emph{infinitary} embedding Ramsey properties.  We tried to make the case that f-ERP is the right generalisation of ERP to arbitrary classes of structures, by showing that both the Kechris-Pestov-Todorcevic correspondence and Scow's theorem transfer immediately when one removes all finiteness assumptions from the class in question and replaces ERP with f-ERP. That being said, it is still unclear what the precise relationship between the two properties is.

We discussed some connections between the two but did not resolve whether inf-ERP coincides with f-ERP.

\begin{question}
    Does f-ERP imply embedding local finiteness? 
\end{question}

By \Cref{inf-ERP iff embedding local fin (assuming countable ERP)}, a positive answer to this question would imply that f-ERP and inf-ERP are equivalent notions, and hence that there is only one natural generalisation of ERP to classes that do not necessarily contain only finite structures. On the other hand, a negative answer to this question, combined with the results from \Cref{sec:local-finiteness} would mean that f-ERP is the ``correct'' generalisation of ERP to arbitrary classes of structures.

\bibliographystyle{alpha}
\bibliography{bibliography}

\end{document}